\RequirePackage{fix-cm}
\documentclass[smallextended,numbook]{svjour3}       
\smartqed  
\usepackage{graphicx}
\usepackage{a4wide}
\usepackage{amsmath}
\usepackage{amssymb}
\usepackage{color}
\usepackage{url}

\newcommand{\N}{\mathbb{N}}

\newcommand{\R}{\mathbb{R}}

\newcommand{\dom}{{\mathrm{dom}}}
\newcommand{\norm}[1]{\left\Vert #1 \right\Vert}

\renewcommand{\H}{{\cal H}}
\newcommand{\V}{{\cal V}}

\newcommand{\A}{{\cal A}}

\renewcommand{\div}{\mathrm{div}\,}

\newcommand{\argmin}{\mathrm{arg}\min}
\newcommand{\st}{\,:\,}

\newcommand{\dx}{\,\mathrm{d}x}
\renewcommand{\d}{\,\mathrm{d}}

\newcommand{\calN}{\mathcal{N}}

\newcommand{\tex}{{T_\mathrm{ex}}}
\newcommand{\sg}{\zeta}
\newcommand{\change}[1]{\textcolor{red}{#1}}
\renewcommand{\change}[1]{#1}
\newcommand{\secchange}[1]{\textcolor{red}{#1}}
\renewcommand{\secchange}[1]{#1}

\journalname{Journal of Evolution Equations}
\begin{document}

\title{Asymptotic Profiles of Nonlinear Homogeneous Evolution Equations of Gradient Flow Type}


\author{Leon Bungert         \and
        Martin Burger 
}


\institute{L. Bungert \and M. Burger \at
              Department Mathematik, Universit\"{a}t Erlangen-N\"{u}rnberg, \\
              Cauerstrasse 11, 91058 Erlangen, Germany. \\
              \email{\{leon.bungert,martin.burger\}@fau.de}           
}

\date{Received: date / Accepted: date}

\maketitle

\begin{abstract}
This work is concerned with the gradient flow of absolutely $p$-homogeneous convex functionals on a Hilbert space, which we show to exhibit finite ($p<2$) or infinite extinction time ($p \geq 2$). We give upper bounds for the finite extinction time and establish \change{sharp} convergence rates of the flow. Moreover, we study next order asymptotics and prove that asymptotic profiles of the solution are eigenfunctions of the subdifferential operator of the functional. To this end, we compare with solutions of an ordinary differential equation which describes the evolution of eigenfunction under the flow. Our work applies, for instance, to local and nonlocal versions of PDEs like $p$-Laplacian evolution equations, the porous medium equation, and fast diffusion equations, herewith generalizing many results from the literature to an abstract setting.

We also demonstrate how our theory extends to general homogeneous evolution equations which are not necessarily a gradient flow. Here we discover an interesting integrability condition which characterizes whether or not asymptotic profiles are eigenfunctions.

\end{abstract}

\section{Introduction}
This work studies the fine asymptotic behavior of the abstract gradient flow
\begin{align}\label{gradflow}\tag{GF}
\begin{cases}
\partial_t u+\partial J(u)\ni 0,\\
u(0)=f.
\end{cases}
\end{align}
Here $J:\H\to\R\cup\{\infty\}$ is an absolutely $p$-homogeneous convex functional on a real Hilbert space~$\H$, and $f\in\H$ is an initial datum (see Section~\ref{sec:setup} for precise definitions).

There is a variety of partial differential equations to which our theory applies and for some of which similar issues have been studied before in a specialized setting. The most prominent and most studied equations are the parabolic $p$-Laplacian equations for $p\geq 1$ with the total variation flow as special case for $p=1$
$$\partial_tu-\div\left(|\nabla u|^{p-2}\nabla u\right)=0,\quad p\geq 1.$$
These equations can also be studied as a fourth order gradient flow in $H^{-1}$, i.e.,
$$ \partial_tu-\Delta\left[\div\left(|\nabla u|^{p-2}\nabla u\right)\right]=0,\quad p\geq 1. $$
Another class of examples are the fast diffusion equations for $1<p<2$, the linear heat equation for $p=2$, and the porous medium equation for $p>2$, i.e.$$
\partial_tu-\Delta u^{p-1}=0,\quad p>1,
$$
which, complemented with suitable boundary conditions, can also be interpreted as Hilbert space gradient flows (cf.~\cite{littig2015porous} for the porous medium / fast diffusion case). Furthermore, as long as homogeneity is preserved, our general model covers \emph{non-local} versions of the equations above, as well. Remarkably, we can also address an eigenvalue problem similar to that of the $\infty$-Laplacian operator \cite{juutinen1999eigenvalue,portilheiro2013degenerate} with our framework. To this end we set $J(u)=\norm{\nabla u}_\infty$ for $u\in W^{1,\infty}\cap L^2$ and $J(u)=\infty$ else, which meets all our assumptions under sufficient regularity of the domain.

The main objective of this work is to prove that asymptotic profiles of the gradient flow \eqref{gradflow} are eigenfunctions of the subdifferential operator $\partial J$. By an asymptotic profile we refer to a suitably rescaled version of the actual solution $u(t)$ of the gradient flow. More precisely, we look for a rescaling $a(t)$ such that $u(t)/a(t)$ converges to some $w_*$ as $t$ tends to the extinction time of the flow (respectively $t \rightarrow \infty$). Here $w_*$ is an eigenfunction of $\partial J$, meaning that $\lambda w_*\in\partial J(w_*)$ for some $\lambda\in\R$ and by ``extinction time'' we refer to the (finite or infinite) time where the solution of the gradient flow stops changing, meaning $\partial_t u(t)=0$ (respectively the minimal time such that that $J(u(t)) =0$).

The rescaling is chosen in such a way that it amplifies the shape of $u(t)$ immediately before it reaches the state of lowest energy as described by the functional $J$. Furthermore, it should be noted that eigenfunctions of $\partial J$ are \emph{self-similar} in the sense that they only shrink under the gradient flow \eqref{gradflow} without changing their shape. 

If the energy is a quadratic form associated to an linear operator with compact and self-adjoint inverse---like it is the case for the Dirichlet energy and the Laplace operator, for instance---elementary spectral theory provides an explicit separated-variable solution of the gradient flow \eqref{gradflow} in terms of the eigenfunctions of the operator. This is already sufficient to accurately describe the asymptotics and study asymptotic profiles. 

If, however, $\partial J$ is a non-linear and potentially multi-valued operator, the situation becomes more challenging since there is typically no basis of eigenfunctions available. Still there is a vast amount of literature dealing with the asymptotic behavior of certain partial differential equations. \secchange{The first results which relate asymptotic profiles of a gradient flow to the eigenvalue problem $\lambda u \in\partial J(u)$ were given in \cite{andreu2002some,andreu2004parabolic} for the case of total variation flow.} Other works deal with $p$-Laplacian equations \cite{kamin1988fundamental,andreu2008nonlocal,portilheiro2013degenerate,vazquez2018asymptotic}, porous medium equations \cite{vazquez2004dirichlet,stan2018porous}, fast diffusion equations \cite{berryman1980stability,bonforte2012behaviour,bonforte2019sharp}, and other PDEs \cite{giga2010nonlinear,blanc2019evolution}, the list far from being exhaustive. A common property of solutions to all this equations appears to be that asymptotically they behave like eigenfunctions of the associated operator. Another observation is that some of the equations above have a finite extinction time whereas others do not. 

Remarkably few efforts have been taken in the literature to transfer above-noted observations to a general PDE like \eqref{gradflow} without specifying the energy $J$ explicitly. In \cite{bungert2019nonlinear} the absolutely 1-homogeneous case was treated in full generality and, for a subclass of functionals, explicit characterizations and relations of asymptotic profiles and the extinction time were proven. Apart from this, the authors were only able to find one more reference \cite{ghidaglia1991exact} which takes a similar route. Here, $J$ is assumed to be {locally sub-homogeneous} of degree $p\geq 2$ and the author proves exact convergence rates to a minimizer of the energy and that asymptotic profiles are eigenfunctions. While this contribution is already very general and captures a wide range of homogeneity degrees $p$, it still lacks important regime $1\leq p<2$ which applies inter alia to $p$-Laplace equations including the total variation flow and fast diffusion equations. Furthermore, the relation between asymptotic profiles and so-called \emph{ground states}, i.e., eigenfunctions with minimal eigenvalue which are especially important in applications like graph clustering (see \cite{buhler2009spectral,bungert2019computing}, for instance), is not illuminated. 

For a general study of the asymptotic behavior of the abstract gradient flow \eqref{gradflow}, only few assumptions are needed. \change{Assuming a \emph{coercivity condition} on $J$ suffices to prove upper rates of convergence, as already shown in \cite{hauer2017kurdyka}, which provides a very general picture on decay rates of gradient flows in metric spaces. In applications, this coercivity is typically provided through a functional inequality of Poincar\'{e}, Sobolev, or H\"{o}lder type. To show that these rates of convergence are actually sharp, one needs to use the \emph{homogeneity} of~$J$.} To ensure that the resulting asymptotic profiles are indeed eigenfunctions, one needs a \emph{compactness} condition on the domain of $J$ which is, however, not very restrictive in most scenarios and seems to be unavoidable, judging from the other existing approaches to prove similar statement. 

In our outlook section on evolution equations of general homogeneous operators we give a condition for the existence of asymptotic profiles which can be used to address more complex PDEs like for example the doubly nonlinear equation 
$$\partial_tu-\div(|u|^{m-1}|\nabla u|^{p-2}\nabla u)=0$$
or the evolutionary Monge-Amp\`{e}re equation
$$\partial_t u-\det(D^2u)=0$$
for which existence of asymptotic profiles has already been established by other techniques in \cite{savare1994asymptotic} and \cite{sanchez2018asymptotic}, respectively. See also \cite{stan2013asymptotic} for a doubly nonlinear equation related to the $p$-Laplacian operator and \cite{le2017eigenvalue} for the Monge-Amp\`{e}re eigenvalue problem.

\subsection{Main Contributions and Outline}

These are the \emph{main contributions} of our work: we characterize finite and infinite extinction times of \eqref{gradflow} and prove exact rates of convergence of solutions by means of a thorough dissipation analysis. We identify a nonlinear Rayleigh quotient as dissipation rate which allows us to study second-order dissipation. Those convergence rates we use to define rescaled solutions which we prove to be eigenfunctions of the associated subdifferential operator. Finally, we give a sneak peek to the asymptotic behavior of semi-groups generated by a general homogeneous operator. 

Furthermore, we would like to point out that our theory answers an open problem raised in \cite{vazquez2016dirichlet} regarding the existence of asymptotic profiles of the fractional $p$-Laplacian evolution equation for $1<p<2$. The statements in Section~\ref{sec:finite_time_ext} fully apply to this equation since it has a gradient flow structure and, hence, asymptotic profiles are eigenfunctions of the fractional $p$-Laplacian operator if enough compactness is provided by the initial conditions (cf.~Example~\ref{ex:existence} below).
\\

The plan of this paper is as follows: we give a very concise overview of the necessary definitions and properties related to the absolutely $p$-homogeneous functionals, their subdifferential, and the associated gradient flow in Section~\ref{sec:setup}. For the interested reader, more detailed statements are collected in the Appendix. Section~\ref{sec:main_sec} is the main part of this work and starts with the study of the gradient flow if the datum $f$ is already an eigenfunction in Section~\ref{sec:evolution_of_eigenfcts}. We address nonlinear eigenvalues and the connection to coercivity in Section~\ref{sec:eigenvalues}. The main ingredients of our analysis are prepared in Sections~\ref{sec:dissipation} and~\ref{sec:times_and_rates} where we study the dissipation of the flow, extinction times, and convergence rates. Sections~\ref{sec:finite_time_ext} and~\ref{sec:infinite_time_ext}, respectively, use these results to prove convergence to eigenfunctions in the cases of finite and infinite extinction time. In Section~\ref{sec:criteria} we study general conditions which assure existence and uniqueness of asymptotic profiles. In Section~\ref{sec:operator} we give a concise outlook on how our theory carries over to general semi-groups of homogeneous operators and we point out limitations. We conclude with some open problems and future working directions.

\subsection{Setup and Assumptions}
\label{sec:setup}

We consider the gradient flow \eqref{gradflow} where $f\in\H$ is an initial datum and $J:\H\to\R\cup\{\infty\}$ is a convex, lower semi-continuous, and proper functional with dense domain, which is absolutely $p$-homogeneous with respect to some $p\geq1$, meaning that 
\begin{equation}\label{eq:p-hom_J}
\begin{split}
J(c u)&=|c|^pJ(u),\quad\forall c\neq 0,\;u\in\H, \\ 
J(0)&=0.
\end{split} 
\end{equation}
Furthermore,
\begin{align}\label{def:subdiff}
\partial J(u)=\left\lbrace \sg\in\H\st J(u)+\langle\sg,v-u\rangle\leq J(v),\;\forall v\in\H\right\rbrace
\end{align}
denotes the (potentially set-valued) subdifferential of $J$ and can be seen as set-valued nonlinear operator $\partial J:\H\rightrightarrows\H$ or as a relation in $\H\times\H$. 

From the definition of the subdifferential \eqref{def:subdiff} and the homogeneity \eqref{eq:p-hom_J} of $J$ it is easy to see that $\partial J$ is $(p-1)$-homogeneous in the sense that
\begin{align}\label{eq:p-1-hom_subdiff}
\partial J(c u)=c|c|^{p-2}\partial J(u),\quad\forall u\in\H,\, c\neq 0.
\end{align}
Furthermore, it is well-known and straightforward to prove \cite{yang2008generalized} that 
\begin{align}\label{eq:euler}
\langle \sg,u\rangle=pJ(u)
\end{align}
holds for all $\sg\in\partial J(u)$ which is a generalization of Euler's homogeneous function theorem to a non-smooth and convex setting. 

Existence and uniqueness of solutions to the abstact evolution equation \eqref{gradflow} follow from Brezis' theory of maximally monotone evolution equations \cite{brezis1973ope}, see Theorem~\ref{thm:brezis}. In the appendix we have collected several important relations that we make use of when calculating with solutions of \eqref{gradflow}.

To increase readability we assume throughout this work that the datum $f$ is orthogonal to the null-space of the functional $J$, which we denoted by $f\in\calN(J)^\perp$. As shown in the appendix, this does not restrict generality since the null-space component of $f$ is invariant under the gradient flow. The main consequence of this assumption is that the solution of the gradient flow converges to zero as $t\to\infty$ instead of converging to its null-space component. For instance, in the case of $J(u)=\int_\Omega|\nabla u|^p\dx$ this means that we restrict ourselves to data with zero mean, i.e., $\int_\Omega f\d x=0$. Furthermore, to avoid triviality we assume $f\neq 0$. Altogether we use the compact notation $f\in\H_0$ where 
\begin{align}
\H_0:=\calN(J)^\perp\setminus\{0\}\subset\H
\end{align}
denotes the set of initial conditions. Again we stress that demanding $f\in\H_0$ does not at all restrict generality as shown in the appendix.

Finally, we define the extinction time of the gradient flow \eqref{gradflow} as
\begin{align}\label{def:ext_time}
\tex:=\tex(f):=\inf\left\lbrace T>0\st u(t)=0\;\forall t\geq T,\,u(t)\text{ solves }\eqref{gradflow} \right\rbrace\in(0,\infty].
\end{align}

\section{Analysis of the Asymptotic Behavior}
\label{sec:main_sec}

In the following we characterize the asymptotic behaviour of gradient flows for $p$-homogeneous functionals. We start by studying the behaviour of nonlinear eigenfunctions during the evolution and then proceed to provide some general relations on dissipation, Rayleigh quotients, and eigenvalues. Combining those ideas with further estimates allows us to characterize extinction times and asymptotic profiles, with some differences for $p<2$ and $p\geq 2$ that we work out in detail.

\subsection{Evolution of Eigenfunctions}
\label{sec:evolution_of_eigenfcts}
In this section we will study the evolution of eigenfunctions of the operator $\partial J$ under the gradient flow \eqref{gradflow}. This will provide us with enough intuition to investigate the asymptotic behavior of general initial data. 

Let us therefore assume that the datum $f$ is already an eigenfunction, i.e., $\lambda f\in\partial J(f)$ for some $\lambda>0$. For such data the unique solution of \eqref{gradflow} has separated variables \cite{cohen2018shape} and is given by $u(t)=a(t)_+f$ for $t\geq0$ where $a(t)$ solves the initial value problem
\begin{align}\label{eq:characteristic_ODE}
a'(t)=-\lambda a(t)^{p-1},\quad a(0)=1,\tag{ODE$_\lambda$}
\end{align}
being parametrized by $\lambda>0$ and the subscript $+$ denotes the positive part of a function. This equation is analytically solvable with 
\begin{align}\label{eq:solution_ODE}
a(t)=
\begin{cases}
(1-(2-p)\lambda t)^\frac{1}{2-p},\quad&p\neq2,\\
\exp(-\lambda t),\quad&p=2.
\end{cases}
\end{align}
and hence one obtains
\begin{align}\label{eq:u_eigenvector}
u(t)=
\begin{cases}
\max\{1-(2-p)\lambda t,0\}^\frac{1}{2-p}f,\quad&p\neq2,\\
\exp(-\lambda t)f,\quad&p=2.
\end{cases}
\end{align}
From these expressions one can draw two conclusions: firstly, for $p\in[1,2)$ one has a finite extinction time, being given by 
\begin{align}\label{eq:extinction_time}
\tex=\frac{1}{(2-p)\lambda},
\end{align}
whereas for $p\geq2$ this is not the case, a behavior which is well-known for the heat equation, for instance. Moreover, in any case a suitable rescaling is given by $a(t)$ since this leads to $u(t)/a(t)\equiv f$ for all $0\leq t<\tex$ which is an eigenfunction.

The main purpose of this work is to show that these two central properties indeed remain true if the eigenfunction $f$ is replaced by some general data. This involves showing that the norms of solutions of the gradient flow behave exactly as $a(t)$. The main ingredient of our analysis will be the following dissipation lemma which we state already at this point since it applies to both the case of finite and infinite extinction time.

\begin{lemma}[Dissipation]\label{lem:dissipation}
Let $u(t)$ solve the gradient flow \eqref{gradflow} with $f\in\H_0$ and let $a(t)$ solve \eqref{eq:characteristic_ODE} with $\lambda>0$. Letting $w(t):=u(t)/a(t)$ it holds for almost all $t\in(0,\tex)$ 
\begin{align}
&\frac{\d}{\d t}J(w(t))+\frac{1}{a(t)^{2-p}}\norm{\frac{\sg(t)}{a(t)^{p-1}}-\lambda w(t)}^2=\lambda\frac{\d}{\d t}\frac{1}{2}\norm{w(t)}^2,\label{eq:diss_p-hom}\\
\lambda&\frac{\d}{\d t}\frac{1}{2}\norm{w(t)}^2=\frac{\lambda}{a(t)^2}\left[\lambda\frac{\norm{u(t)}^2}{a(t)^{2-p}}-pJ(u(t))\right].\label{eq:dnormw}
\end{align}
\end{lemma}
\begin{proof}
The proof follows from simple differentiation calculus and using \eqref{eq:characteristic_ODE}. 
\qed 
\end{proof}

Next, we shortly explain the relation between the existence of a positive minimal eigenvalue of $\partial J$ and coercivity of $J$ which will be prerequisite for the rest of our analysis.

\subsection{Nonlinear Eigenvalues and Coercivity}
\label{sec:eigenvalues}
Throughout this work we assume that $J$ is coercive in the sense that 
\begin{align}\label{ineq:coercivity}
\norm{u}^p\leq CJ(u),\;\forall u\in\H_0,
\end{align}
where $C>0$ is a constant. It is obvious that this inequality is equivalent to positive lower bound of the Rayleigh quotient associated with $J$, i.e.,
\begin{align}\label{eq:first_eigenvalue}
\lambda_1:=\inf_{u\in\H_0}\frac{pJ(u)}{\norm{u}^p}>0,
\end{align}
and that the optimal constant in \eqref{ineq:coercivity} is given by $C=p/\lambda_1$. Minimizers of the nonlinear Rayleigh quotient 
\begin{align}\label{def:rayleigh}
R(u):=\frac{pJ(u)}{\norm{u}^p}
\end{align}
over $\H_0$ are referred to as \emph{ground states}, in the sequel, and are only determined up to multiplication with a scalar. 

Remember that we have called $u\in\H$ an eigenfunction if $\lambda u\in\partial J(u)$ for some number~$\lambda$. However, the value $\lambda$ does not have a meaningful interpretation as eigenvalue since according to \eqref{eq:p-1-hom_subdiff} the inclusion $\lambda u\in\partial J(u)$ is invariant under multiplication with a scalar \emph{only if} $p=2$. Hence, we introduce a scale-invariant definition of nonlinear eigenvalues which takes the homogeneity of $J$ into account.

\begin{definition}[Nonlinear \change{eigenfunctions and their eigenvalues}]\label{def:rigorous_eigenfunctions}
We say that $u\in\H\setminus\{0\}$ is an eigenfunction of $\partial J$ with eigenvalue $\lambda\in\R$ if 
\begin{align}\label{eq:scale-inv_egenvalues}
\lambda \norm{u}^{p-2}u\in\partial J(u).
\end{align}
\end{definition}

\secchange{%
\begin{remark}[Other definitions of nonlinear eigenfunctions]
Note that in the literature there are also other definitions of nonlinear eigenfunctions, which are often related to the doubly-nonlinear eigenvalue problem
\begin{align}\label{eq:doubly_nonlinear_eigen}
\lambda q\in\partial J(u),\; q\in\partial H(u),
\end{align}
where $H$ is another convex functional. For instance, if $J(u)=\int_\Omega|\nabla u|^p\dx$ and $H(u)=\int_\Omega|u|^p\dx$ for $1\leq p<\infty$, this yields the standard definition of $p$-Laplacian eigenfunctions
$$\lambda|u|^{p-2}u=-\Delta_p u,$$
where $|.|$ denotes the absolute value and $\Delta_p$ is the $p$-Laplacian. Note that this differs from our definition of eigenfunctions given in \eqref{eq:scale-inv_egenvalues}, which arises by choosing $H(u)=\frac1p \norm{u}^p$ in \eqref{eq:doubly_nonlinear_eigen} with a Hilbert norm $\norm{\cdot}$. However, in the context of asymptotic profiles of gradient flows, Definition~\ref{def:rigorous_eigenfunctions} is the right choice, as our results show.
\end{remark}}

Having this definition at hand, we can connect eigenfunctions to critical points of the Rayleigh quotient in \eqref{eq:first_eigenvalue}. The statement is somewhat standard so we omit the proof.

\begin{proposition}[Critical points of the Rayleigh quotient]\label{prop:critical_rayleigh}
It holds that $u\in\H\setminus\{0\}$ is an eigenfunction with eigenvalue $\lambda$ if and only if $u$ is a critical point of the Rayleigh quotient \eqref{def:rayleigh}. In this case the eigenvalue is given by $\lambda=R(u)$. In particular, ground states are eigenfunctions with eigenvalue $\lambda_1$ defined in~\eqref{eq:first_eigenvalue}.
\end{proposition}

Let us collect some more properties of nonlinear eigenfunctions and their eigenvalues which are similar to the well-known linear case. 

\begin{proposition}\label{prop:eigenvalues}
Let $u\in\H\setminus\{0\}$ be an eigenfunction of $\partial J$ with eigenvalue $\lambda>0$ in the sense of Definition~\ref{def:rigorous_eigenfunctions}. Then it holds
\begin{enumerate}
\item $u\in\H_0$ and $\lambda\geq\lambda_1$,
\item $cu$ is an eigenfunction with eigenvalue $\lambda$ for all $c\in\R$.
\end{enumerate}
\end{proposition}
\begin{proof}
Ad 1.: Proposition~\ref{prop:orth} implies $u\in\H_0$. According to \eqref{eq:euler} and \eqref{eq:scale-inv_egenvalues} it holds $\lambda\norm{u}^p=pJ(u)$ for eigenfunctions $u$, which implies the statement. Ad 2.: Due to \eqref{eq:p-1-hom_subdiff} the inclusion \eqref{eq:scale-inv_egenvalues} is invariant under multiplication with scalars. 
\qed 
\end{proof}

Obviously, it is possible to compute the eigenvalue of elements $u\in\H$ which meet $\lambda u\in\partial J(u)$ for some $\lambda$ by simple rescaling.

\begin{corollary}\label{cor:rescaling_eigenvalues}
Any $u\in\H_0$ that meets ${\lambda} u\in\partial J(u)$ for some ${\lambda}\geq 0$ is an eigenfunction with eigenvalue $\tilde{\lambda}:={\lambda}/\norm{u}^{p-2}$ and vice versa.
\end{corollary}

\subsection{Dissipation Analysis}
\label{sec:dissipation}

While it is generally true that solutions of the gradient flow of a convex functional approach an energy minimum monotonously, in the $p$-homogeneous case the identity \eqref{eq:diss_u}, which states that $\norm{u(t)}^2$ decreases with rate proportional to $J(u(t))$, sharpens this statement. However, in \cite{ghidaglia1991exact,varvaruca2004exact} it has been observed that one should rather consider the dissipation of $\norm{u(t)}^{p-2}$ instead which turns out to be proportional to the Rayleigh quotient \eqref{def:rayleigh} evaluated along the gradient flow.  

The fundamental importance of coercivity (or equivalently a positive lower bound for the Rayleigh quotient) for us lies in the fact that it ensures that $\norm{u(t)}$ decreases \emph{sufficiently fast} and gives rise to an upper rate. In contrast, by studying the dissipation of the dissipation, i.e., the derivative of the Rayleigh quotient, one can infer that this rate is actually \emph{sharp}, meaning that the flow does not extinct too quickly. For a compact notation we define the map
\begin{align}\label{eq:Lambda}
t\mapsto\Lambda(t):=R(u(t))=\frac{pJ(u(t))}{\norm{u(t)}^p}
\end{align}
for $0<t<\tex$ where $u(t)$ is the solution of the gradient flow \eqref{gradflow} with $f\in\H_0$.
\begin{remark}
It is obvious that $\Lambda$ is well-defined for all $0<t<\tex$. Furthermore, from Theorem~\ref{thm:brezis} together with the identity \eqref{eq:diss_u} one can infer that it is locally Lipschitz continuous and hence differentiable for almost all $t\in(0,\tex)$.
\end{remark}

An important property of the map $\Lambda$ is that it quantifies the dissipation of the flow, as the following Proposition states.

\begin{proposition}\label{prop:diss_u_Lambda}
Let $u(t)$ solve \eqref{gradflow} with $f\in\H_0$. Then it holds for all $0<t<\tex$ that
\begin{align}\label{eq:diss_u_Lambda}
\begin{cases}
\frac{1}{2-p}\frac{\d}{\d t}\norm{u(t)}^{2-p}=-\Lambda(t),\quad&p\neq 2,\\[10pt]
\phantom{\frac{1}{2-p}}\frac{\d}{\d t}\norm{u(t)}^{2\phantom{-p}}=-2\Lambda(t)\norm{u(t)}^2,\quad&p=2.
\end{cases}
\end{align}
\end{proposition} 
\begin{proof}
Using \eqref{eq:diss_u} and the definition of $\Lambda$ in \eqref{eq:Lambda} we can compute
$$\frac{\d }{\d t}\frac{1}{2}\norm{u(t)}^2=-pJ(u(t))=-\Lambda(t)\norm{u(t)}^p.$$
In the case $p=2$ this already proves \eqref{eq:diss_u_Lambda}. For $p>2$ we can use this equality and calculate
$$\frac{1}{2-p}\frac{\d}{\d t}\norm{u(t)}^{2-p}=\norm{u(t)}^{1-p}\frac{\d}{\d t}\norm{u(t)}=\norm{u(t)}^{-p}\frac{\d}{\d t}\frac{1}{2}\norm{u(t)}^2=-\Lambda(t)$$
which concludes the proof also in this case.
\qed 
\end{proof}

Crucial for obtaining lower bounds for the convergence of $u(t)$ to zero will be that $t\mapsto\Lambda(t)$ is non-increasing. Furthermore, its critical points correspond to eigenfunctions.

\begin{proposition}[Evolution of the Rayleigh quotient]\label{prop:decrease_rayleigh}
Let $u(t)$ solve \eqref{gradflow} with $f\in\H_0$ and let $\Lambda$ be given by \eqref{eq:Lambda}. Then it holds $\Lambda'(t)\leq 0$ for almost all $t\in(0,\tex)$. Furthermore, $\Lambda'(t)=0$ if and only if $u(t)$ is an eigenfunction with eigenvalue $\Lambda(t)$ and $\Lambda$ is differentiable at $t$. 
\end{proposition}
\begin{proof}
Using the \eqref{eq:derivative_J} and the Euler identity \eqref{eq:euler} we can calculate the derivative of $\Lambda$ and find
\begin{align*}
\Lambda'(t)&=-p^2J(u(t))\norm{u(t)}^{-p-1}\frac{\d}{\d t}\norm{u(t)}-p\norm{u(t)}^{-p}\norm{\sg(t)}^2\\
&=-p^2J(u(t))\norm{u(t)}^{-p-2}\frac{\d}{\d t}\frac{1}{2}\norm{u(t)}^2-p\norm{u(t)}^{-p}\norm{\sg(t)}^2\\
&=p\norm{u(t)}^{-p}\left[\frac{\langle\sg(t),u(t)\rangle^2}{\norm{u(t)}^2}-\norm{\sg(t)}^2\right]\leq 0,
\end{align*}
where the inequality follows from Cauchy-Schwarz. Hence, $\Lambda$ is non-increasing. Furthermore, $\Lambda'(t)=0$ if and only if $\sg(t)=\lambda u(t)$ for some $\lambda\in\R$. Assuming this to hold, one obtains $\lambda u(t)=\sg(t)\in\partial J(u(t))$, and hence $u(t)$ is an eigenfunction and the eigenvalue is given by $R(u(t))=\Lambda(t)$ according to Proposition~\ref{prop:critical_rayleigh}. Conversely, assume that $u(t)$ is an eigenfunction, meaning that $\lambda u(t)\in\partial J(u(t))$ for some $\lambda>0$. Then it was shown\footnote{Strictly speaking, it was only shown in the absolutely 1-homogeneous case. However, the proof for general $p$ works precisely the same.} in \cite{bungert2019nonlinear} that $\lambda u(t)$ has minimal norm in the convex set $\partial J(u(t))$. One the other hand, Theorem~\ref{thm:brezis} states that also $\sg(t)$ has minimal norm in $\partial J(u(t))$. Since minimal norm elements in convex sets are unique, this implies $\lambda u(t)=\sg(t)$ and hence $\Lambda'(t)=0$ if the derivative exists.
\qed 
\end{proof}

\subsection{Extinction Time, Lower and Upper Bounds}
\label{sec:times_and_rates}
Having studied the dissipation of $t\mapsto\norm{u(t)}$ along the gradient flow, we are now able to investigate extinction times and, correspondingly, convergence rates of $u(t)$ as $t\to\tex$ in the Hilbert norm.

We start with proving \emph{infinite extinction time} in the case $p\geq 2$ and deriving lower bounds on the convergence.

\begin{theorem}[Infinite extinction time for $p\geq 2$]\label{thm:infinite_ext}
Let $u(t)$ solve \eqref{gradflow} with $f\in\H_0$ and $p\geq 2$. Then it holds $\tex=\infty$.
\end{theorem}
\begin{proof}
The proof is very similar to \cite[Thm.~2.1]{ghidaglia1991exact}. Since the authors only deal with the case $p>2$ and $J(u)=\frac{1}{p}\int_\Omega|\nabla u|^p\d x$ there, we give the full proof of this results, for the matter of completeness.

Let us assume there is a finite and positive extinction time $T$, chosen minimally. Propositions~\ref{prop:diss_u_Lambda} and~\ref{prop:decrease_rayleigh} imply that for any $0<\delta\leq t<T$ it holds
\begin{align*}
\begin{cases}
\frac{1}{2-p}\frac{\d}{\d t}\norm{u(t)}^{2-p}\geq-\Lambda(\delta),\quad&p\neq 2,\\[10pt]
\phantom{\frac{1}{2-p}}\frac{\d}{\d t}\norm{u(t)}^{2\phantom{-p}}\geq-2\Lambda(\delta)\norm{u(t)}^2,\quad&p=2.
\end{cases}
\end{align*}
If $p=2$ we can use Gronwall's Lemma to infer 
$$\norm{u(t)}^2\geq\norm{u(\delta)}^2\exp(-2\Lambda(\delta)t)>0$$ 
which yields a contradiction for $t\to T$. If $p>2$ we integrate from $\delta$ to $t<T$ which yields
$$\norm{u(t)}^{2-p}\leq\norm{u(\delta)}^{2-p}+\Lambda(\delta)(p-2)(t-\delta).$$
Letting $t\to T$ gives the contradiction since the left hand side diverges whereas the right hand side remains bounded.
\qed 
\end{proof}

The proof of Theorem~\ref{thm:infinite_ext} has already constructed lower bounds for $\norm{u(t)}$ in the case $p\geq2$ which we collect in the following Corollary. Note that already in \cite{alikakos1982lower} lower bounds in the case $p>2$ and a Fr\'{e}chet differentiable functional $J$ where given. 

\begin{corollary}[Lower bounds for infinite extinction]\label{cor:lower_bounds}
Under the conditions of Theorem~\ref{thm:infinite_ext} and letting $\Lambda(t)$ be given by \eqref{eq:Lambda} it holds for all $\delta>0$ and $t\geq\delta$
\begin{alignat}{2}
\norm{u(t)}^{2\phantom{p-}}&\geq\norm{u(\delta)}^2\exp\left(-2\Lambda(\delta)(t-\delta)\right),\qquad&&p=2,\\
\norm{u(t)}^{p-2}&\geq\frac{1}{\norm{u(\delta)}^{2-p}+(p-2)\Lambda(\delta)(t-\delta)},\qquad&&p>2,
\end{alignat}
where $\delta=0$ is admissible if $J(f)<\infty$.
\end{corollary}

Now we turn to general upper bounds of the evolution which follow from the coercivity assumption \eqref{eq:first_eigenvalue}. 

\begin{corollary}[General upper bounds]\label{cor:upper_bounds}
Let $u(t)$ solve \eqref{gradflow} with $f\in\H_0$. Then it holds
\begin{subequations}\label{ineq:decay_u_gen}
\begin{alignat}{2}
\norm{u(t)}^{2-p}&\leq\norm{f}^{2-p}-(2-p)\lambda_1 t,\qquad &&p<2,\label{ineq:decay_u_p<2}\\
\norm{u(t)}^{2\phantom{-p}}&\leq\norm{f}^2\exp(-2\lambda_1t),\qquad &&p=2,\label{ineq:decay_u_p=2}\\
\norm{u(t)}^{p-2}&\leq\frac{1}{\norm{f}^{2-p}+(p-2)\lambda_1 t},\qquad &&p>2.\label{ineq:decay_u_p>2}
\end{alignat}
\end{subequations}
\end{corollary}
\begin{proof}
For the proof one uses $\Lambda(t)\geq\lambda_1$ in \eqref{eq:diss_u_Lambda} and integrates the resulting inequality from $0$ to $t$ (using Gronwall's Lemma for $p=2$).
\qed 
\end{proof}

\begin{remark}[Upper bounds for metric gradient flows]\label{rem:metric_gradient_flow}
In \cite{hauer2017kurdyka} the authors proved such upper bounds relying on solely a coercivity assumption of the type \eqref{eq:first_eigenvalue} in the very general setting of metric gradient flows. Furthermore, they relate coercivity to the validity of a Kurdyka-\L{}ojasiewicz inequality. From these results and also from our proofs it is clear that only coercivity and no homogeneity of the functional is necessary for upper bounds of the evolution. 
\end{remark}

\change{%
\begin{remark}[The role of homogeneity]
Note that without using the homogeneity of~$J$, equality \eqref{eq:diss_u_Lambda} turns into $\leq$ which however does not affect the proof of the upper bounds in Corollary~\ref{cor:upper_bounds}. However, the lower bounds massively rely on Proposition~\ref{prop:decrease_rayleigh} for which homogeneity seems to be unavoidable.
\end{remark}}

The upper bound for $p<2$ implies that the solution extincts in finite time and yields an estimate of the extinction time in terms of the first eigenvalue.
\begin{theorem}[Finite extinction time for $p<2$]\label{thm:extinction_time}
The extinction time $\tex$ of the gradient flow \eqref{gradflow} with $f\in\H_0$ and $p<2$ can be bounded as
\begin{align}\label{ineq:upper_bound_ext_time}
\tex\leq\frac{\norm{f}^{2-p}}{(2-p)\lambda_1}<\infty.
\end{align}
\end{theorem}
\begin{proof}
Inequality \eqref{ineq:upper_bound_ext_time} follows directly from \eqref{ineq:decay_u_p<2}. 
\qed 
\end{proof}

Up to now we have not seen lower bounds in the case $p<2$. These---together with sharper upper bounds---follow a-posteriori from the existence of a finite extinction time.

\begin{corollary}[Sharper bounds for finite extinction]
Let $u(t)$ solve \eqref{gradflow} with $f\in\H_0$ and $p<2$. Then it holds
\begin{align}
\norm{u(t)}^{2-p}&\leq (2-p)\Lambda(t)(\tex-t),\label{ineq:upper_bound_p<2}\\
\norm{u(t)}^{2-p}&\geq (2-p)\lambda_1(\tex-t).\label{ineq:lower_bound_p<2}
\end{align}
\end{corollary}
\begin{proof}
For \eqref{ineq:upper_bound_p<2} one integrates \eqref{eq:diss_u_Lambda} form $t$ to $\tex$ and uses that $\Lambda$ is non-increasing. For the lower bound \eqref{ineq:lower_bound_p<2} one uses $\Lambda(t)\geq\lambda_1$ in \eqref{eq:diss_u_Lambda} and integrates the resulting inequality from $t$ to $\tex$.
\qed 
\end{proof}

\begin{remark}
Some remarks regarding the coercivity \eqref{ineq:coercivity} and finite extinction time are in order.
\begin{enumerate}
\item If $\H$ is finite dimensional and $J$ is the $p$-th power of a semi-norm, the equivalence of norms show that for all $u\in\H_0$ it holds $cJ(u)\leq\norm{u}^p\leq CJ(u)$ for some $c,C>0$ such that one obtains lower and upper bounds from \eqref{eq:diss_u}, straightforwardly, without studying the evolution of the Rayleigh quotient.
\item If $\dom(J)\subset\H$ is a Banach space, inequality \eqref{ineq:coercivity} implies an continuous embedding of $\dom(J)$ into $\H$, so for instance if $J(u)=\int_\Omega|\nabla u|^p\d x$ and $\H=L^2(\Omega)$ this is the embedding $W^{1,p}(\Omega)\hookrightarrow L^2(\Omega)$.
\item All the proofs show that in fact the validity of \eqref{ineq:coercivity} is only required for the solution $u(t)$ of the gradient flow, not for all $u\in\H_0$. In particular, for special initial conditions one can have coercivity of $J$ \emph{along} the solution of the gradient flow, even if this is not true on the whole space. As a prototypical example, one can think of the parabolic $p$-Laplace equation where the ambient space dimension is too large to have a compact or continuous embedding of $W^{1,p}$ or $BV$ into $L^2$. However, for bounded initial data $f$ one can use the maximum principle to show that coercivity is indeed fulfilled along the flow (cf.~\cite{andreu2002some} for the total variation flow).  
\item An inequality similar to \eqref{ineq:coercivity} is indeed necessary for finite extinction: assume that $\tex$ is the finite extinction time of \eqref{gradflow}. Then by integrating \eqref{eq:diss_u} and using that $t\mapsto J(u(t))$ is non-increasing it follows 
$$\norm{u(t)}^2=2p\int_t^\tex J(u(s))\d s\leq 2p(\tex-t)J(u(t)),\quad\forall 0\leq t\leq\tex.$$
\end{enumerate}

\end{remark}

\subsection{Case I: Finite Extinction Time}
\label{sec:finite_time_ext}
Our aim is to prove that the solution of the initial value problem in \eqref{eq:characteristic_ODE} is a suitable rescaling for $u(t)$ in the regime $1\leq p<2$. If the data is an eigenvector, we have already seen in Section~\ref{sec:evolution_of_eigenfcts} that $\lambda$ in \eqref{eq:characteristic_ODE} must be chosen as the eigenvalue. For general data, however, one has to replace the eigenvalue by a different number. Motivated by \eqref{eq:extinction_time} we define 
\begin{align}\label{eq:lambda}
\lambda&:=\frac{1}{(2-p)\tex},\\
a(t)&\text{ solves \eqref{eq:characteristic_ODE},}\notag\\
\label{eq:rescaling_p-hom}
w(t)&:=\frac{u(t)}{a(t)},\qquad 0\leq t<\tex.
\end{align}
In order to prove that for general data the rescalings \eqref{eq:rescaling_p-hom} converge to an eigenfunction, the main ingredient will be the convergence 
\begin{align}\label{eq:convergence_p_w}
\frac{\sg(t)}{a(t)^{p-1}}-\lambda w(t)\to 0,\quad t\to\tex,
\end{align}
which we will derive from the dissipation Lemma~\ref{lem:dissipation}.

We start with bounds from below and above of the rescaled solutions.

\begin{lemma}[Bounds for rescaled solutions]\label{lem:boundedness_resc}
The map $t\mapsto{w(t)}$ defined in \eqref{eq:rescaling_p-hom} meets
\begin{align}\label{ineq:bounds_p<2}
(2-p)\lambda_1\tex\leq\norm{w(t)}^{2-p}\leq(2-p)\tex\Lambda(t),\quad t\geq 0,
\end{align}
and is, in particular, bounded.
\end{lemma}
\begin{proof}
This follows directly from \eqref{ineq:upper_bound_p<2}, \eqref{ineq:lower_bound_p<2}, and the definition of $a(t)$.
\qed 
\end{proof}

The boundedness of $t\mapsto\norm{w(t)}$ would already be sufficient to prove the convergence \eqref{eq:convergence_p_w}. However, we can show a bit more, namely that the norms $\norm{w(t)}$ are in fact non-increasing.

\begin{lemma}\label{lem:decreasing_profiles}
It holds that $t\mapsto\norm{w(t)}$ is non-increasing.
\end{lemma}
\begin{proof}
We already have an expression for the derivative of the map $t\mapsto\frac{1}{2}\norm{w(t)}^2$, given in \eqref{eq:dnormw}. Studying the square brackets there and using the sharp upper bounds \eqref{ineq:upper_bound_p<2} we derive
\begin{align*}
\lambda\frac{\norm{u(t)}^2}{a(t)^{2-p}}-pJ(u(t))\leq\lambda\frac{(2-p)}{a(t)^{2-p}}\Lambda(t)\norm{u(t)}^p(\tex-t)-pJ(u(t))=0,
\end{align*}
where we also used the definitions of $a(t)$, $\Lambda(t)$, and $\lambda$ (cf.~\eqref{eq:solution_ODE}, \eqref{eq:Lambda}, and \eqref{eq:lambda}, respectively). This already concludes the proof.
\qed 
\end{proof}

Dissipation and boundedness from Lemmas~\ref{lem:dissipation} and~\ref{lem:boundedness_resc} imply the desired asymptotics for the subgradients $\sg(t)$:
\begin{corollary}
The convergence \eqref{eq:convergence_p_w} holds true.
\end{corollary}
\begin{proof}
Integrating inequality \eqref{eq:diss_p-hom} from $s$ to $t$ where $0<s<t<\tex$ and using Lemma~\ref{lem:decreasing_profiles} yields
$$J(w(t))-J(w(s))+\int_s^t\frac{1}{a(\tau)^{2-p}}\norm{\frac{\sg(\tau)}{a(\tau)^{p-1}}-\lambda w(\tau)}^2\d\tau\leq 0.$$
Letting $t$ tend to $\tex$, using the non-negativity of $J$ and the fact that $a(t)^{2-p}\to 0$ as $t\to\tex$ we can conclude that the squared norm in the integral must converge to zero as $\tau\to\tex$ which finishes the proof.
\qed 
\end{proof}

Lemmas \ref{lem:boundedness_resc} or \ref{lem:decreasing_profiles} also imply that (up to a subsequence) the rescalings $w(t)$ weakly converge to some $w_*$ in $\H$ which we refer to as \emph{asymptotic profile}. However, in order to prove that $w_*$ is an eigenfunction, one needs the convergence to be strong which is an additional regularity assumption that, for instance, can be assured through suitable initial conditions $f$ of the gradient flow (cf.~Section~\ref{sec:criteria}). 

\begin{theorem}[Asymptotic profiles for finite extinction]\label{thm:ext_prof}
Assume that $w(t)$ converges (possibly up to a subsequence) strongly to some $w_*\in\H$ as $t\to\tex$. Then it holds $\lambda w_*\in\partial J(w_*)$, $w_*\neq 0$, and $\norm{w_*}\leq\norm{f}$.
\end{theorem}
\begin{proof}
Lemmas~\ref{lem:boundedness_resc} and \ref{lem:decreasing_profiles} imply that $w_*\neq0$ and $\norm{w_*}\leq\norm{f}$, respectively. To prove that $\lambda w_*\in\partial J(w_*)$ we start with
$$J(u(t))+\langle\sg(t),v-u(t)\rangle\leq J(v),\quad\forall v\in\H,\,t>0,$$
which is an equivalent formulation of the gradient flow \eqref{gradflow}. Using this together with the definition of $w(t)$ and the $p$-homogeneity of $J$, we infer
\begin{align*}
J(w(t))+\left\langle\frac{\sg(t)}{a(t)^{p-1}},v-w(t)\right\rangle 
&=\frac{1}{a(t)^p}\left[J(u(t))+\left\langle \sg(t),a(t)v-u(t)\right\rangle\right]\\
&\leq\frac{1}{a(t)^p}J(a(t)v)=J(v),\quad\forall v\in\H.
\end{align*}
If we now use the lower semi-continuity of $J$, the strong convergence of $w(t)$ to $w_*$, and the convergence \eqref{eq:convergence_p_w}, we can pass to the limit in this inequality to obtain
$$J(w_*)+\left\langle\lambda w_*,v-w_*\right\rangle\leq J(v),\quad\forall v\in\H,$$ 
which proves $\lambda w_*\in\partial J(w_*)$.
\qed 
\end{proof}

\begin{remark}
In principle, asymptotic profiles are not uniquely determined, i.e., the sequence $w(t)$ can have several accumulation points as $t\to\tex$.
\end{remark}

Let us now turn to the question under which conditions a possible asymptotic profile is even a ground state. It turns out that this is the case if and only if the lower bound in \eqref{ineq:bounds_p<2} for the rescaled solutions is asymptotically sharp.

\begin{proposition}[Ground states as asymptotic profiles]\label{prop:char_ground_state_1}
For any asymptotic profile it holds $(2-p)\lambda_1\tex\leq \norm{w_*}^{2-p}$ and equality is true if and only if $w_*$ is a ground state.
\end{proposition}
\begin{proof}
The statement follows by evaluating $\lambda_1\leq J(w_*)/\norm{w_*}^p$, using that $\lambda w_*\in\partial J(w_*)$, and the definition of $\lambda$ in \eqref{eq:lambda}.
\qed 
\end{proof}

It is obvious from \eqref{ineq:bounds_p<2} that if $\Lambda(t)\to\lambda_1$ as $t\to\tex$ then the lower bound is asymptotically sharp and hence $w_*$ is a ground state. The converse is also true.

\begin{corollary}\label{cor:char_ground_state_II}
An asymptotic profile $w_*$ is a ground state if and only if $\lim_{t\nearrow\tex}\Lambda(t)=\lambda_1$.
\end{corollary}
\begin{proof}
Integrating \eqref{eq:diss_u_Lambda} from $0<t<\tex$ to $\tex$ yields
$$\norm{u(t)}^{2-p}=(2-p)\int_t^\tex\Lambda(s)\d s.$$
Dividing both sides by $a(t)^{2-p}$ we obtain
$$\norm{w(t)}^{2-p}=\frac{(2-p)\tex}{\tex-t}\int_t^\tex\Lambda(s)\d s.$$
Letting $t\to\tex$ implies due to the continuity of $\Lambda$ on $(0,\tex)$ and the strong convergence of $w(t)$ that
$$\norm{w_*}^{2-p}=(2-p)\tex \lim_{t\nearrow\tex}\Lambda(s)\d s.$$
Proposition~\ref{prop:char_ground_state_1} concludes the proof.
\qed 
\end{proof}

Furthermore, we can use the characterization above to infer that the extinction time of the gradient flow attains its upper bound from \eqref{ineq:upper_bound_ext_time} then any asymptotic profile is a ground state.

\begin{corollary}
Let $\tex=\frac{\norm{f}^{2-p}}{(2-p)\lambda_1}$ and let $w_*$ be an asymptotic profile. Then $w_*$ is a ground state.
\end{corollary}
\begin{proof}
The fact that $\norm{w_*}\leq\norm{f}$ together with Proposition~\ref{prop:char_ground_state_1} yields 
$$\norm{w_*}^{2-p}\leq\norm{f}^{2-p}=(2-p)\lambda_1\tex\leq\norm{w_*}^{2-p}.$$
Hence, equality holds which by Proposition~\ref{prop:char_ground_state_1} implies that $w_*$ is a ground state.
\qed 
\end{proof}

The converse of this statement is general false which we will see later. First we show that it is very exotic that the gradient flow has maximal extinction time since this requires the initialization to be a ground state.

\begin{proposition}
Assume that $\tex=\frac{\norm{f}^{2-p}}{(2-p)\lambda_1}$ and an asymptotic profile exists. Then it holds $\Lambda(t)=\lambda_1$ for all $0\leq t<\tex$ and, hence, $f$ is a ground state.
\end{proposition}
\begin{proof}
From the previous corollary we infer that any asymptotic profile $w_*$ is a ground state and, hence, it holds according to Proposition~\ref{prop:char_ground_state_1} that $\tex=\frac{\norm{w_*}^{2-p}}{(2-p)\lambda_1}$. However, since $\norm{w_*}\leq\norm{f}$ holds and $t\mapsto\norm{w(t)}$ is non-increasing, we see that the extinction time is maximal if and only if $\norm{w(t)}=\norm{f}$ for all $t\geq 0$. From the proof of Lemma~\ref{lem:decreasing_profiles} we can infer that this is the case only if the upper bound \eqref{ineq:upper_bound_p<2} is sharp for all $t\geq 0$, meaning that
$$\norm{u(t)}^{2-p}=(2-p)\Lambda(t)(\tex-t).$$
This can be rewritten as
$$\norm{w(t)}^{2-p}=(2-p)\tex\Lambda(t)$$
and since the left hand side is constant, we end up with
$$\norm{f}^{2-p}=(2-p)\tex\Lambda(t)$$
which, finally, implies that $\Lambda(t)=\lambda_1$ for all $0\leq t<\tex$. Hence, the Rayleigh quotient of the datum $f$ is given by $R(f)=\Lambda(0)=\lambda_1$ which means that $f$ is a ground state.
\qed 
\end{proof}

Judging from these results one might think that the asymptotic profile can only be a ground state if the gradient flow is already initialized with a ground state. Fortunately, this is not the case as the following example shows. Also in general, there are situations where asymptotic profiles are always ground states, independent of the initialization (cf.~Remark~\ref{rem:ground_states_positivity} below).

\begin{example}[The spectral 1-homogeneous case]
We consider the absolutely 1-homogeneous case which was thoroughly investigated in \cite{bungert2019nonlinear}. Using the results there we can construct an example where the asymptotic profile is a ground state although the extinction time is not maximal. To this end, let $f=g+h$ be the datum where $g$ and $h$ have unit norm and satisfy $\lambda g\in\partial J(g)$ and $\lambda_1 h\in\partial J(h)$ with $\lambda>\lambda_1>0$. Furthermore, assume that $\langle g,h\rangle = 0$ and that $\lambda g+\lambda_1 h\in\partial J(0)$. For example, these assumptions apply to the total variation flow initialized with two calibrable sets \cite{alter2005characterization} which are non-overlapping. Under these conditions it was shown in \cite{bungert2019nonlinear} (see also \cite{schmidt2018inverse,bungert2019solution}) that the solution of the gradient flow has an explicit form given by
$$u(t)=
\begin{cases}
(1-\lambda t)g+(1-\lambda_1 t)h,\quad &0\leq t\leq\frac{1}{\lambda},\\
(1-\lambda_1 t)h,\quad &\frac{1}{\lambda}< t \leq\frac{1}{\lambda_1}.
\end{cases}
$$
Hence, we observe that $\tex=1/\lambda_1$ and that the rescaled solution $w(t)=u(t)/(1-\lambda_1 t)$ converges to $w_*=h$ which is a ground state. However, the extinction time does not attain its upper bound since
$$\tex=\frac{1}{\lambda_1}<\frac{\sqrt{2}}{\lambda_1}=\frac{\sqrt{\norm{g}^2+\norm{h}^2}}{\lambda_1}=\frac{\norm{f}}{\lambda_1}.$$

\end{example}

\subsection{Case II: Infinite Extinction Time}
\label{sec:infinite_time_ext}
Now we will turn to the case of infinite extinction time where it is less straightforward to find a suitable rescaling since -- due to the lack of a finite extinction time -- one has less information about the quantitative decay speed of solutions. In \cite{varvaruca2004exact} a rescaling of the form $u(t)/\norm{u(t)}$ was proposed and it was shown that (up to a subsequence) this converges to a non-trivial eigenfunction. However, since we are interested in the connection of asymptotic profiles and ground states (i.e., minimizers of \eqref{eq:first_eigenvalue}), we pursue a rescaling strategy which, similarly to the scenario of finite time extinction treated in Section~\ref{sec:finite_time_ext}, is based on \eqref{eq:characteristic_ODE}.

To make sure that the rescaling does not decay to quickly, which might lead to a blow up of the rescaled solutions, it is necessary to choose a rescaling with a sufficiently slow decay among the solutions of \eqref{eq:characteristic_ODE}.

We start with an illustrative example and consider the heat equation $\partial_tu(t,x)=\Delta u(t,x)$ with homogeneous Neumann conditions, which is the $L^2$-gradient flow of the Dirichlet energy $J(u)=\frac{1}{2}\int_\Omega|\nabla u|^2\d x$ and hence corresponds to a 2-homogeneous functional. Technical details aside, we assume that there is a orthonormal basis of eigenfunctions $(v_i)_{i\in\N}$ of the negative Laplace operator with eigenvalues $0<\lambda_1<\lambda_2<\dots$, meaning that $\lambda_i v_i=-\Delta v_i$ for all $i\in\N$. Straightforwardly, this gives rise to an explicit solution of the heat equation with initial data $f\in L^2(\Omega)$ as $u(t,x)=\sum_ic_i\exp(-\lambda_it)v_i(x)$, where $c_i=\langle f,v_i\rangle_{L^2(\Omega)}$ and we assume $c_1\neq 0$ for simplicity. Consequently, the rescaling $a(t)=\exp(-\lambda t)$ leads to
$$u(t,x)/a(t)\to c_1v_1(x),\quad t\to\infty,$$
if and only if $\lambda=\lambda_1$. For $\lambda>\lambda_1$ the rescaled solutions would blow up whereas they would converge to zero for $\lambda<\lambda_1$. Note that if $c_1=0$, meaning that the datum is orthogonal to the ground state $v_1$, the asymptotic profile is zero, as well. Nevertheless, the solution of $a'(t)=-\lambda_1 a(t),\;c(0)=1,$ seems to be a suitable rescaling if one is interested in asymptotic profiles which have the chance to be ground states. Also in the general $p$-homogeneous case with $p>2$ one would like to be able to isolate ground states. Hence, we study the case of a datum $f$ which fulfills $\lambda f\in\partial J(f)$ and take $a(t)$ as solution to \eqref{eq:characteristic_ODE} with $\lambda=\lambda_1\norm{f}^{p-2}$. According to \eqref{eq:u_eigenvector} this leads to 
\begin{align}\label{eq:extinct_eigenfunction}
\frac{u(t)}{a(t)}=\left(\frac{1+(p-2)\lambda_1\norm{f}^{p-2}t}{1+(p-2)\lambda t}\right)^\frac{1}{p-2}f\to
\left(\frac{\lambda_1\norm{f}^{p-2}}{\lambda}\right)^\frac{1}{p-2}f,\quad t\to\infty.
\end{align}
Now since according to Corollary~\ref{cor:rescaling_eigenvalues} it holds $\lambda\geq\lambda_1\norm{f}^{p-2}$, we see that the asymptotic profile yields back $f$ if this is a ground state, and a multiple with smaller norm else. Note that this behavior is less degenerate -- but also less discriminating -- than for $p=2$ since a non-zero asymptotic profile can be achieved even if the datum is orthogonal to all ground states. 

\begin{remark}[Other possible rescalings]
Except in the degenerate case $p=2$, any rescaling of the form $a(t)=t^{-\frac{1}{p-2}}$ guarantees that that the rescaled solutions $w(t)=u(t)/a(t)$ converge to eigenfunctions. However, in order to have a consistent presentation for all cases $p\geq 2$, we choose $a(t)$ as a suitable solution of \eqref{eq:characteristic_ODE} which has the same qualitative rate of decrease.
\end{remark}

As already indicated, our precise rescaling set-up for the case of infinite extinction time is the following:
\begin{align}
\lambda&:=\lambda_1\norm{f}^{p-2},\label{eq:lambda_inf}\\
a(t)&\text{ solves \eqref{eq:characteristic_ODE}},\notag\\
\label{eq:rescales_solutions_p>2}
w(t)&:=\frac{u(t)}{a(t)},\qquad t\geq 0.
\end{align}

The techniques to prove that rescaled solutions converge to eigenfunctions are analogous to the case of finite extinction time and build mainly on the dissipation Lemma~\ref{lem:dissipation} and boundedness.

\begin{lemma}[Bounds for rescaled solutions]\label{lem:bounds_p>2}
For $p>2$ and for all $\delta>0$ the map $t\mapsto\norm{w(t)}$ meets 
\begin{align}\label{ineq:bounds_p>2}
\frac{\lambda_1\norm{u(\delta)}^{p-2}}{\Lambda(\delta)}\leq\norm{w(t)}^{p-2}\leq\norm{f}^{p-2},\quad t\geq\delta,
\end{align}
where $\delta=0$ is admissible if $J(f)<\infty$. If $p=2$ the map is non-increasing and, in particular, it holds $\norm{w(t)}\leq\norm{w(0)}=\norm{f}$ for all $t\geq0$.
\end{lemma}
\begin{proof}
For $p>2$ the bounds follow from Corollaries~\ref{cor:lower_bounds} and ~\ref{cor:upper_bounds} together with the definition of $a(t)$ in \eqref{eq:solution_ODE}. If $p=2$ we can use \eqref{eq:dnormw} and the fact that $\lambda=\lambda_1$ to obtain
\begin{align*}
\frac{\d}{\d t}\frac{1}{2}\norm{w(t)}^2=\frac{1}{a(t)^2}\left[\lambda\norm{u(t)}^2-\langle\sg(t),u(t)\rangle\right]\leq \frac{1}{a(t)^2}\left[2J(u(t))-2J(u(t))\right]=0,
\end{align*}
for almost all $t>0$ which implies that $t\mapsto\norm{w(t)}$ is non-increasing.
\qed 
\end{proof}

As before, dissipation plus the boundedness of the rescaled solutions imply the correct asymptotics for the subgradients $\sg(t)$:
\begin{corollary}
It holds $\frac{\sg(t)}{a(t)^{p-1}}-\lambda w(t)\to 0$ strongly as $t\to\infty$.
\end{corollary}
\begin{proof}
Integrating \eqref{eq:derivative_J} yields
$$J(w(t))-J(w(s))+\int_s^t\frac{1}{a(\tau)^{2-p}}\norm{\frac{\sg(\tau)}{a(\tau)^{p-1}}-\lambda w(\tau)}^2\d\tau=\frac{\lambda}{2}\Big(\norm{w(t)}-\norm{w(s)}\Big).$$
If we now let $t\to\infty$, use the boundedness from Lemma~\ref{lem:bounds_p>2}, and the fact that $1/a(\tau)^{2-p}$ does not tend to zero as $\tau\to\infty$, we can infer the statement.
\qed 
\end{proof}

From the boundedness of $\norm{w(t)}$ as $t\to\infty$ it follows again that, up to a subsequence, $w(t)$ converges weakly to some $w_*\in\H$ as $t\to\infty$. For $w_*$ to be an eigenfunction, again, strong convergence is required.

\begin{theorem}[Asymptotic profiles for infinite extinction]
Assume that $w(t)$ converges (possibly up to a subsequence) strongly to some $w_*\in\H$ as $t\to\infty$. Then it holds $\lambda w_*\in\partial J(w_*)$ and $w_*\neq 0$ if $p>2$.
\end{theorem}
\begin{proof}
The proof works exactly as the one of Theorem~\ref{thm:ext_prof}. Furthermore, the lower bound in~\eqref{ineq:bounds_p>2} implies that $w_*\neq 0$.
\qed 
\end{proof}

\begin{remark}
We have already seen in the introductory examples of this section that for $p=2$ one cannot expect to have a non-trivial asymptotic profile, in general. Corollary~\ref{cor:lower_bounds} may state that $\norm{u(t)}$ does not decrease faster than an exponential, however, the estimate is very pessimistic since for $0<\delta\leq t$ it implies
$$\norm{w(t)}^2=\frac{\norm{u(t)}^2}{a(t)^2}\geq\norm{u(\delta)}^2\exp(\delta\Lambda(\delta))\exp\left(-2\left[\Lambda(\delta)-\lambda_1\right]t\right) $$
and the right hand side always tends to zero as $t\to\infty$ if $f=u(0)$ was not a ground state already which implies $\Lambda(t)=\lambda_1$ for all $t>0$.
\end{remark}

Next we investigate when $w_*$ is a ground state. Opposed to the case of finite extinction time, this is the case if and only if the upper bound \eqref{ineq:bounds_p>2} for the rescaled solutions is asymptotically sharp.

\begin{proposition}
If $p=2$ then $w_*$ is a non-trivial ground state or $w_*=0$. For $p>2$ it holds $\norm{w_*}\leq\norm{f}$ with equality if and only if $w_*$ is a ground state.
\end{proposition}
\begin{proof}
The statement follows by evaluating $\lambda_1\leq J(w_*)/\norm{w_*}^p$, using that $\lambda w_*\in\partial J(w_*)$, and the definition of $\lambda$ in \eqref{eq:lambda_inf}.
\qed 
\end{proof}

As in the case of finite extinction time, we get
\begin{corollary}
For $p>2$ an asymptotic profile $w_*$ is a ground state if and only if $\lim_{t\to\infty}\Lambda(t)=\lambda_1$.
\end{corollary}

\begin{remark}
It seems tempting to always use 2-homogeneous functionals in order to compute ground states since according to the previous proposition their asymptotic profile is either zero or a non-trivial ground state. This, however, does not do the trick as the following argument shows. Assume that $f$ is an eigenfunction but no ground state of the 1-homogeneous functional $J$, meaning that $\lambda\frac{f}{\norm{f}}\in\partial J(f)$ with $\lambda>\lambda_1$. With the chain rule it is easy to see that $\lambda^2 f\in\partial\tilde{J}(f)$ where $\tilde{J}(\cdot)=J(\cdot)^2/2$ is absolutely 2-homogeneous and convex. Furthermore, the minimum of the Rayleigh quotient with respect to $\tilde{J}$ is given by $\tilde{\lambda}_1=\lambda_1^2$. Hence, the solution of the gradient flow of $\tilde{J}$ is given by $\tilde{u}(t)=\exp(-\lambda^2 t)f$ and the corresponding rescalings $\tilde{w}(t)=\tilde{u(t)}/\exp(-\tilde{\lambda_1}t)$ converge to zero since $\lambda>\lambda_1$.   
\end{remark}

\subsection{Criteria for Existence and Uniqueness of Asymptotic Profiles}
\label{sec:criteria}
In this section we briefly collect some sufficient conditions which ensure the existence of asymptotic profiles, i.e., the strong convergence of a subsequence of $w(t)$ to some $w_*$ as $t\to\tex$, as well as uniqueness of the profile, meaning that the whole sequence converges to a unique profile. 

\begin{theorem}[Existence]\label{thm:existence}
Asymptotic profiles exist if one of the following scenarios is met:
\begin{enumerate}
\item $\H$ is finite-dimensional
\item $\dom(J)$ is a Banach space which is compactly embedded in $\H$ 
\item $\dom(J)$ is a Banach space which is compactly embedded in some Banach space $\V$. Furthermore, the gradient flow \eqref{gradflow} provides enough regularity such that strong convergence of $w(t)$ in $\V$ implies strong convergence in $\H$.
\end{enumerate}
\end{theorem} 

\begin{example}\label{ex:existence}
The second item in Theorem~\ref{thm:existence} applies for instance if $\dom(J)=W^{1,p}(\Omega)$ under sufficient regularity of the domain $\Omega\subset\R^n$. Here one has the compact embedding $W^{1,p}(\Omega)\Subset L^q(\Omega)$ for all $q<np/(n-p)$  such that for $p\in(2n/(n+2),n)$ one has a compact embedding in $\H=L^2(\Omega)$. If however $p$ is too small or one even has the case $\dom(J)=BV(\Omega)$, the third item applies. One can employ a maximum principle or $L^q$-regularity of the corresponding flow together with suitable initial conditions, to infer that the compact embedding into a space larger than $L^2(\Omega)$ is in fact enough to guarantee the strong convergence in $L^2(\Omega)$ (cf.~\cite{andreu2002some} for the total variation flow).
\end{example}

For a uniqueness proof, we have to assume a partial order on the Hilbert space. Not that with uniqueness we \emph{do not} mean that every initial datum $f$ gives rise to the same asymptotic profile. Instead we state a condition which ensures that the gradient flow with respect to a fixed initial datum yields at most one asymptotic profile, meaning that the rescaled solutions $w(t)$ do not possess more than one accumulation point as $t\to\tex$.

\begin{theorem}[Uniqueness of the profile]\label{thm:uniqueness}
Assume that $(\H,\geq)$ is a partially ordered space and the initial datum meets $f\geq0$. Then the solution $u(t)$ of \eqref{gradflow} has at most one asymptotic profile.
\end{theorem}
\begin{proof}
Crucial is that $u(t)\geq 0$ for all $t\geq 0$. The proof techniques for $p\leq 2$ and $p>2$ are somewhat different. We begin with the latter case. A classical result for homogeneous evolution equations \cite{crandall1980regularizing} states that for positive solutions one has
$$\partial_tu(t)\geq-\frac{1}{p-2}\frac{u(t)}{t}.$$
This estimate immediately implies that for $\tilde{w}(t):=t^\frac{1}{p-2}u(t)$ it holds 
$$\partial_t\tilde{w}(t)=\frac{1}{p-2}t^{\frac{1}{p-2}-1}u(t)+t^\frac{1}{p-2}\partial_t u(t)\geq 0.$$
Hence, $\tilde{w}(t)$ increases monotonously and, thus, has at most one limit $\tilde{w}_*$. Since, however, 
$$\norm{w(t)-c\tilde{w}(t)}\to 0,\quad t\to\infty$$ 
for a suitable $c=c(p,\lambda)$, one obtains that $w(t)$ has a unique limit $w_*=c\tilde{w}_*$, as well.

In the case $p\leq 2$ we utilize a simple transformation in order to exploit the uniqueness of the asymptotic profile in the case $p>2$. Without loss of generality, we assume that $f\in\dom(\partial J)$ which implies that $t\mapsto J(u(t))$ is globally Lipschitz continuous on the interval $(0,\infty)$. We let $q>2p\geq 2$ and define $r:=q/p>2$. Then the functional $\mathcal{J}(u):=J(u)^r/r$ is absolutely $q$-homogeneous, convex, and lower semi-continuous. Furthermore, by application of the chain rule, one can easily see that the solution of the gradient flow with respect to $\cal J$ and datum $f$ is given by $\tilde{u}(\tau)=u(\varphi(\tau))$, where $u$ solves the gradient flow with respect to $J$ and datum $f$ and the time reparametrization $\varphi$ is defined by the initial value problem
$$
\begin{cases}
\varphi'(\tau)&=J(u(\varphi(\tau)))^{r-1},\quad\tau>0,\\
\varphi(0)&=0.
\end{cases}
$$
Existence and uniqueness of $\varphi$ on the positive reals follow from standard arguments utilizing that $t\mapsto J(u(t))^{r-1}$ is globally Lipschitz continuous. To see this it is sufficient to note that this map is Lipschitz continuous on $[0,
\tex]$ and zero on $(\tex,\infty)$. Furthermore, $\tau\mapsto\varphi(\tau)$ is invertible on the interval $(0,\tex)$ since it is strictly increasing, and it holds $\varphi^{-1}(t)\to\infty$ as $t\to\tex$. The latter statement follows from the fact that there is no $\tau>0$ such that $\varphi(\tau)=\tex$ since otherwise $\tilde{u}$ would have finite extinction time which is impossible according to Theorem~\ref{thm:infinite_ext}. 

From the first part of the theorem we know that the rescalings of $\tilde{u}$, given by 
$$\tilde{w}(\tau)=\frac{\tilde{u}(\tau)}{(1-(2-q)\lambda_q \tau)^\frac{1}{2-q}}$$
where $\lambda_q$ is defined by \eqref{eq:lambda_inf} for $p=q$, have at most one limit and without loss of generality we assume that $\tilde{w}(\tau)$ converges to some $\tilde{w}_*$ as $\tau\to\infty$. Consequently, we have for $t\in(0,\tex)$ that
\begin{align}\label{eq:transformation_profiles}
w(t)=\frac{u(t)}{a(t)}=\frac{\tilde{u}(\varphi^{-1}(t))}{(1-(2-q)\lambda_q \varphi^{-1}(t))^\frac{1}{2-q}}\frac{(1-(2-q)\lambda_q \varphi^{-1}(t))^\frac{1}{2-q}}{a(t)}
\end{align}
where $a(t)$ is given by \eqref{eq:rescaling_p-hom}. Now the first term on the right hand side converges to $\tilde{w}_*$ since $\varphi^{-1}(t)\to\infty$ as $t\to\tex$. Furthermore, we know from Lemmas~\ref{lem:decreasing_profiles} and \ref{lem:bounds_p>2} that $\norm{w(t)}$ is decreasing and hence has a unique limit as $t\to\tex$. This implies that the second term in \eqref{eq:transformation_profiles}, and consequently also $w(t)$, has a unique limit as $t\to\tex$. 
\qed 
\end{proof}

\begin{remark}[Ground states and positivity]\label{rem:ground_states_positivity}
The concept of positivity can also be used to investigate whether asymptotic profiles are ground state. If, for instance, it is know that all eigenfunctions apart from the ground state change sign, then any asymptotic profile coincides with a ground state if the initialization is chosen non-negative. This applies, for instance, to the $p$-Laplacian operators ($1<p<\infty$) with homogeneous Dirichlet boundary conditions (cf.~\cite{kawohl2006positive}, for instance). If, in addition, the ground state is known to be simple, meaning unique up to scalar multiplication, then we even have uniqueness of the asymptotic profile irrespective of the chosen initialization $f\geq 0$. Hence, asymptotic profiles generalize other methods for finding ground states of nonlinear operators like \cite{bozorgnia2016convergence}, for instance.
\end{remark}

\section{Evolution Equations with  Monotone Operators}
\label{sec:operator}
In this section, we give an outlook on how our results for homogeneous gradient flows can be transferred to evolution equations governed by a general homogeneous operator, which is not necessarily the subdifferential of an homogeneous functional. 

More precisely, we study the more general evolution equation
\begin{align}\label{op_flow}\tag{OF}
\begin{cases}
\partial_t u(t)+\A u(t)\ni 0,\\
u(0)=f,
\end{cases}
\end{align}
where $\A:\H\rightrightarrows\H$ is a maximally monotone and $(p-1)$-homogeneous operator for $p\in[1,\infty)$, meaning that $\A(cu)=c|c|^{p-2}\A u$ holds for $u\in\H$ and $c\in\R\setminus\{0\}$. In analogy to \eqref{eq:subgrad_minimal_norm}, we define the single-valued operator
\begin{align}\label{eq:operator_minimal_norm}
\A^0:\H\to\H,\quad u\mapsto\A^0u:=\argmin\{\norm{v}\st v\in\A u\}
\end{align}
and remark that the solution $u(t)$ of \eqref{op_flow} is right-differentiable for every $t>0$ with $\partial_t^+u(t)=-\A^0u(t)$, just as in Theorem~\ref{thm:brezis} which dealt with the case $\A=\partial J$.

However, a fundamental difference to the gradient flow theory is that, in order to ensure well-posedness of \eqref{op_flow} by means of Brezis' theory \cite{brezis1973ope}, the initial datum $f$ is assumed to lie in the domain of $\A$ given by $\dom(\A):=\left\lbrace u\in\H\st\A u\neq\emptyset\right\rbrace$. This poses already a severe restriction since typically the domain of $\A$ is a proper subset of the ambient Hilbert space $\H$. However, in the case $p=1$ it was shown recently in \cite{hauer2019regularizing} that \eqref{op_flow} is well-defined for general $f\in\H$. 

If one assumes that the operator $\A^0$ fulfills a coercivity condition of the form 
\begin{align}\label{eq:first_eigenvalue_operator}
\inf_{u\in\H}\frac{\langle\A^0 u,u\rangle}{\norm{u}^p}=\lambda_1>0,
\end{align}
which is the analogue to \eqref{eq:first_eigenvalue} in the gradient flow case, one infers by similar computations as in Section~\ref{sec:times_and_rates} that the solution of \eqref{op_flow} converges to zero with explicit upper bounds and finite extinction times. 

For a compact notation we define a surrogate function 
\begin{align}\label{eq:surrogate}
h(t):=\langle\A^0 u(t),u(t)\rangle,\quad t>0,
\end{align}
which fulfills $h(t)=pJ(u(t))$ in the case of a gradient flow of a absolutely $p$-homogeneous functional $J$. Since $t\mapsto\A^0u(t)$ is right-continuous and $t\mapsto u(t)$ is Lipschitz continuous, it follows that $h$ is a priori is only right-continuous. However, $h$ can be more regular as it can be seen from the gradient flow case. It is crucial to give sense to $h'(t)$ in order to study the asymptotic behavior of the solutions of \eqref{op_flow}.

Remember that the main ingredient of our gradient flow analysis was the dissipation Lemma~\ref{lem:dissipation}, which studied the evolution of $t\mapsto J(w(t))$ where $w(t)=u(t)/a(t)$ is the rescaled solution. In the present case of an operator without subdifferential structure, we can study the evolution of $H(t):=h(t)/a(t)^p$ which equals $pJ(w(t))$ if $\A=\partial J$. We obtain the following generalization of Lemma~\ref{lem:dissipation}:
\begin{lemma}[Dissipation]\label{lem:dissipation_operator}
Assume that $h$ is differentiable almost everywhere in $(0,\tex)$, $u(t)$ solves \eqref{op_flow} and $a(t)$ is a solution of \eqref{eq:characteristic_ODE} with $\lambda>0$. Letting $w(t):=u(t)/a(t)$ and $H(t):=h(t)/a(t)^p$ it holds for almost all $t\in(0,\tex)$ 
\begin{align}
\frac{1}{p}&\frac{\d}{\d t}H(t)+\frac{1}{a(t)^{2-p}}\norm{\A^0w(t)-\lambda w(t)}^2=\frac{1}{p}\frac{h'(t)+p\norm{\A^0u(t)}^2}{h(t)}H(t)+\lambda\frac{\d}{\d t}\frac{1}{2}\norm{w(t)}^2\\
\lambda&\frac{\d}{\d t}\frac{1}{2}\norm{w(t)}^2=\frac{\lambda}{a(t)^2}\left[\lambda\frac{\norm{u(t)}^2}{a(t)^{2-p}}-\langle\A^0 u(t),u(t)\rangle\right].
\end{align}
\end{lemma}
\begin{proof}
The result follows by simple calculations and uses the $(p-1)$-homogeneity of $\A$.
\qed 
\end{proof}

\begin{corollary}\label{cor:convergence_operator}
Under the conditions of Lemma~\ref{lem:dissipation_operator} assume that
\begin{align}\label{eq:integrability_assmpt}
g(t):=\frac{1}{p}\left(\frac{h'(t)+p\norm{\A^0 u(t)}^2}{h(t)}\right)_+\in L^1((s,\tex)),\qquad\forall 0<s<\tex.
\end{align}
Then it holds
\begin{align}\label{eq:convergence_operator}
\A^0w(t)-\lambda w(t)\to 0
\end{align}
strongly in $\H$, as $t\to\tex$.
\end{corollary}
\begin{proof}
Using the integrability assumption \eqref{eq:integrability_assmpt}, Gronwall's Lemma implies that $H(t)$ has a finite limit as $t\to\tex$. Furthermore, $t\mapsto\norm{w(t)}$ is bounded which follows from \eqref{eq:first_eigenvalue_operator} as before. By integration, convergence \eqref{eq:convergence_operator} follows.
\qed 
\end{proof}

\begin{theorem}[Asymptotic profiles]
Under the conditions of Lemma~\ref{lem:dissipation_operator} assume that $w(t)$ converges (possibly up to a subsequence) strongly to some $w_*\in\H$ as $t\to\tex$ and that $\A^0$ is continuous, meaning that $u_k\to u$ implies $\A^0u_k\to\A^0u$. Then it holds $\lambda w_*=\A^0w_*$.
\end{theorem}
\begin{proof}
The statement follows from the triangle inequality. We compute
$$0\leq\norm{\A^0 w_*-\lambda w_*}\leq\norm{\A^0w_*-\A^0w(t)}+\norm{\A^0w(t)-\lambda w(t)}+\norm{\lambda w(t)-\lambda w_*}.$$
In the limit $t\to\tex$ the first term of the right hand side vanishes due to the continuity of $\A^0$, the second term due to \eqref{eq:convergence_operator}, and the third term tends to zero since we assumed $w(t)$ to converge strongly.
\qed 
\end{proof}

\begin{remark}[Assumptions on the operator]
Note that the results above make in fact no use of the maximal monotonicity of the operator $\A$, which is only required to ensure well-posedness of \eqref{op_flow}. Hence, the results remain true for general homogeneous operators whenever \eqref{op_flow} is well-posed and its solution has the correct decay behavior, meaning the operator is coercive as in \eqref{eq:first_eigenvalue_operator}.
\end{remark}

\begin{remark}[Integrability assumption]\label{rem:integrability}
For gradient flows it is apparent from the last item of Theorem~\ref{thm:brezis} that $h'(t)=-p\norm{\A^0u(t)}^2$, which in particular implies the integrability assumption \eqref{eq:integrability_assmpt} in Corollary~\ref{cor:convergence_operator}. Furthermore, this identity can be used to show that the generalized Rayleigh quotient $\Lambda(t)$ is decreasing. Since the whole of our analysis in Section~\ref{sec:times_and_rates} relies on coercivity and a decreasing Rayleigh quotient, almost all results from Section~\ref{sec:main_sec} remain valid in the operator case, if only $h'(t)=-p\norm{\A^0 u(t)}^2$ holds. For a general operator $\A^0$ using the chain rule this is the case if the Gateaux-derivative of $u\mapsto\langle\A^0 u,u\rangle$ in direction $v$ is given by $p\langle\A^0u,v\rangle.$ However, this identity can in general be false, even for linear operators as we will see below. 
\end{remark}

\begin{remark}[Existence of eigenvalues]\label{rem:existence_Ev}
Unlike in the gradient flow case, the coercivity condition \eqref{eq:first_eigenvalue_operator} \emph{does not} imply the existence of a positive eigenvalue. For instance, the matrix 
$$
\begin{pmatrix}
1 & \alpha \\ 
-\alpha & 1
\end{pmatrix} ,\qquad\alpha\neq 0,
$$
does not posses a real eigenvalue, despite being coercive (and in particular positive definite) in the sense of \eqref{eq:first_eigenvalue_operator}. Consequently the integrability assumption \eqref{eq:integrability_assmpt} cannot be fulfilled in such cases.
\end{remark}

As a first class of non-subdifferential operators for which one can show the integrability condition \eqref{eq:integrability_assmpt} we study the case of $\A$ being a general linear operator on a finite dimensional Hilbert space, meaning that $\A$ can be represented as a matrix. If this matrix is non-symmetric, the flow \eqref{op_flow} is not a gradient flow like \eqref{gradflow}. Note that a non-symmetric and coercive matrix does not necessarily have a real eigenvalue, cf. Remark~\ref{rem:existence_Ev} above. However, we will have to demand that every eigenvalue of $\A$ is real, in order to prove that the integrability condition \eqref{eq:integrability_assmpt} is satisfied and the flow admits an asymptotic profile which is an eigenvector of $\A$.

\begin{example}[Non-symmetric matrix]
Let $\A\in\R^{n\times n}$ be a matrix, let $\A_\mathrm{sym}=(\A+\A^T)/2$ and $\A_\mathrm{asym}=(\A-\A^T)/2$ denote the symmetric and antisymmetric parts of $\A$. Furthermore, we assume that $\A_\mathrm{sym}$ is positively semi-definite which implies the same for $\A$ and hence $\A$ is a monotone operator. The solution of \eqref{op_flow} is then given by $u(t)=\exp(-t\A)f$ such that $h(t)$ is differentiable and it holds
\begin{align*}
h'(t)&=\frac{\d}{\d t}\langle \A u(t),u(t)\rangle\\
&=\langle\A u(t),\partial_t u(t)\rangle+\langle\partial_t u(t),\A^T u(t)\rangle\\
&=-2\norm{\A u(t)}^2+\langle(\A-\A^T)u(t),\A u(t)\rangle.
\end{align*}
Consequently, we are interested in the integrability of
$$g(t)=\frac{1}{2}\left(\frac{h'(t)+2\norm{\A u(t)}^2}{h(t)}\right)_+=\left(\frac{\langle \A_\mathrm{asym}u(t),\A u(t)\rangle}{\langle\A u(t), u(t)\rangle}\right)_+.$$
If $\A_\mathrm{asym}=0$ then $g(t)=0$ which is no surprise since in this case one has a gradient flow with respect to the functional $J(u)=\frac{1}{2}\langle\A u,u\rangle$. In the general case $\A_\mathrm{asym}\neq 0$, however, we have to investigate the integrability of $g(t)$ as $t\to\infty$. 

Let us therefore use the explicit formula for $u(t)$ to transform $g(t)$ to 
\begin{align}\label{eq:g_for_matrices}
g(t)=\left(\frac{\left\langle\left[\A\exp(-t\A)\right]^T\A_\mathrm{asym}\exp(-t\A)f,f\right\rangle}{\langle\A \exp(-t\A)f,\exp(-t\A)f\rangle}\right)_+.
\end{align}

Taking into account that $\A$ should be coercive and possess a positive real eigenvalue, we assume that $\A$ has a Schur form, meaning that there is a orthgonal matrix $U\in\R^{n\times n}$ and a upper triangular matrix $S\in\R^{n\times n}$, which has the eigenvalues of $\A$ on its diagonal, such that $\A=U^TSU$. In fact, due to the structure of $g(t)$ in \eqref{eq:g_for_matrices}, one can assume without loss of generality that 
$$\A=
\begin{pmatrix}
\lambda & \alpha & • & \ast \\ 
0 & \lambda & • & • \\ 
• & • & \ddots & • \\ 
• & • & • & \lambda
\end{pmatrix} 
\qquad\text{and hence}
\qquad
\A_\mathrm{asym}=
\begin{pmatrix}
0 & \alpha & • & \ast \\ 
-\alpha & 0 & • & • \\ 
• & • & \ddots & • \\ 
-\ast & • & • & 0
\end{pmatrix} ,
$$ 
where $\lambda>0$ is the eigenvalue of $\A$, and $\alpha\in\R$ is some number. Due to this simple structure it holds 
\begin{align}\label{eq:exponential_jordan}
\exp(-t\A)\sim e^{-t\lambda}
\begin{pmatrix}
1 & t & \dots & a t^{n-2} & b t^{n-1} \\ 
• & \ddots & \ddots &  & c t^{n-2} \\ 
• & • & \ddots & \ddots & \vdots \\ 
• & • & • & \ddots & t \\ 
• & • & • & • & 1
\end{pmatrix},
\end{align}
where $a,b,c\in\R$ are some numbers. Here we ignored all other coefficients and lower powers of $t$, the most important thing being that the powers increase by one on each superdiagonal. In the following we keep track of the two highest powers only. Consequently, one obtains
\begin{align}
\A\exp(-t\A)&\sim e^{-t\lambda}
\begin{pmatrix}
\ast & \dots & \lambda a t^{n-2} & \lambda b t^{n-1}+\alpha c t^{n-2} \\ 
\vdots & • & \ast & \lambda c t^{n-2} \\ 
\vdots & • & • & \vdots \\ 
\ast & \dots & \dots & \ast
\end{pmatrix},\\[10pt]
\A_\mathrm{asym}\exp(-t\A)&\sim e^{-t\lambda} 
\begin{pmatrix}
\ast  & \dots & \ast & \alpha c t^{n-2} \\ 
\vdots  & • & -\alpha a t^{n-2} & -\alpha b t^{n-1} \\ 
\vdots  & • & • & \vdots \\ 
\ast  & \dots & \dots & \ast
\end{pmatrix}.
\end{align}
Finally, with regard to the numerator in \eqref{eq:g_for_matrices} we compute
$$
\left[\A\exp(-t\A)\right]^T\A_\mathrm{asym}\exp(-t\A)\sim 
e^{-2\lambda t}
\begin{pmatrix}
\ast & \dots & \dots & \ast \\ 
\vdots & • & • & \vdots \\ 
\vdots & • & • & \alpha a c t^{2n-4} \\ 
\vdots & \dots & \dots & (\alpha c)^2 t^{2n-4}
\end{pmatrix},
$$
and we observe that the highest occuring power of $t$ is $2n-4$ and not $2n-2$ as one might expect from \eqref{eq:exponential_jordan}. Let us for simplicity assume that the last entry of $f$ is different from zero. Then the numerator of $g(t)$ is of order $e^{-2\lambda t}t^{2n-4}$ whereas the denominator has order $e^{-2\lambda t}t^{2n-2}$ as $t\to\infty$. Consequently, we obtain that $g(t)=\mathcal{O}(1/t^2)$ and is, in particular, integrable. If $f$ has zero entries, then this is also true since the correct powers of $t$ are annihilated both in the numerator and denominator as one can see similar to the computations above.

We conclude this example remarking that the assumption of having only real eigenvalues, cannot be relaxed. The matrix
$$\A=
\begin{pmatrix}
1 & 0 & 0 \\ 
0 & 1 & 1 \\ 
0 & -1 & 1
\end{pmatrix} 
$$
is coercive, hence positive definite, and has eigenvalues $1$ and $1\pm i$. Nevertheless, it can be computed explicitly that the flow \eqref{op_flow} with respect to $\A$ does not have an asymptotic profile.
\end{example}

Finally, we study the Monge-Amp\`{e}re operator which is no subdifferential and still yields asymptotic profiles. This was already proven in \cite{sanchez2018asymptotic} using comparison principles. Using our abstract framework, however, it is straightforward to prove the integrability condition \eqref{eq:integrability_assmpt} and hence the existence of asymptotic profiles.

\begin{example}[Time-dependent Monge-Amp\`{e}re equation]
In this example we study the Dirichlet problem for the time-dependent Monge-Amp\`{e}re equation, which has the form
\begin{align}\label{eq:monge-ampere}
    \begin{cases}
    \partial_tu-\det(D^2u)=0\quad&\text{in }(0,\infty)\times\Omega,\\
    u=0\quad&\text{in }(0,\infty)\times\partial\Omega,\\
    u(0)=f\quad&\text{in }\Omega,
    \end{cases}
\end{align}
where $\Omega\subset\R^n$ and $f$ are sufficiently smooth and convex to ensure well-posedness (cf.~\cite{sanchez2018asymptotic}). In this example we do not study regularity issues and tacitly assume enough smoothness of the involved functions. The operator $\A$ is given by $\A u=-\det(D^2u)$ and is $p$-homogeneous with $p:=n+1$. With regard to Remark~\ref{rem:integrability} we define the Monge-Amp\`{e}re functional
$$M(u):=\langle\A u,u\rangle=\int_\Omega-u\det(D^2u)\d x$$
which is coercive in the sense of \eqref{eq:first_eigenvalue_operator} according to \cite{le2017eigenvalue}. Furthermore, it was shown in \cite{tso1990real} that its Gateaux-derivative is given by
$$\delta M(u)[v]=(n+1)\int_\Omega -\det(D^2u)v\d x=p\langle\A u,v\rangle,\quad u,v=0\text{ on }\partial\Omega.$$
Correspondingly, the derivative of the surrogate function $h(t)=\langle\A u(t),u(t)\rangle$ where $u(t)$ solves \eqref{eq:monge-ampere} is given by $h'(t)=-p\norm{\A u(t)}^2$ and hence the integrability assumption \eqref{eq:integrability_assmpt} is met with $g(t)=0$ and Remark~\ref{rem:integrability} applies. Hence, we obtain the result that asymptotic profiles of the Monge-Amp\`{e}re flow are eigenfunctions purely from the properties of the operator and associated functional, without using any comparison principle.
\end{example}

\section*{Open Problems}
In the following we list some open problems which are subject to future work:
\begin{itemize}
\item When is the surrogate function $h$ in \eqref{eq:surrogate} differentiable? In general, it is only right-continuous since $t\mapsto\A^0 u(t)$ does not have enough regularity. We think that the pairing $\langle\A^0 u(t),u(t)\rangle$ is differentiable, nevertheless, with the proof pending.
\item Under which conditions on the operator $\A$ does the \emph{integrability assumption} from Corollary~\ref{cor:convergence_operator} hold? Is it alreasy enough to assume that all eigenvalues are positive as in the linear case or is the non-linear scenario more involved?
\item Can the theory from Section~\ref{sec:operator} be transferred to \emph{accretive operators} defined on Banach spaces?
\item How does our theory translate to \emph{doubly nonlinear} evolution equations of the form
$$\partial H(\partial_t u)+\partial J(u)\ni 0,\qquad u(0)=f,$$
defined on Banach spaces, where the eigenvalue problem becomes $\partial J(u)-\lambda\partial H(u)\ni 0$? For $H=\frac{1}{p}\norm{\cdot}^p$ and $J$ both being absolutely $p$-homogeneous with $p\in(1,\infty)$ this was answered in \cite{hynd2017approximation}. These results basically coincide with our result in the 2-homogeneous case: asymptotic profiles are ground states or zero, which seems to be a typical behavior for equal homogeneities. For the general case of $H$ being absolutely $q$-homogeneous with $q>1$ and $J$ being absolutely $p$-homogeneous with $p\geq 1$ it is straightforward to show that for eigenfunctions fulfilling $\partial J(f)-\lambda\partial H(f)\ni 0$ the flow admits a separate variable solution of the form $u(t)=\max\{a(t),0\}f$ with the decay profile 
$$a(t)=
\begin{cases}
\left(1-\frac{q-p}{q-1}\lambda^\frac{1}{q-1}t\right)^\frac{q-1}{q-p},\quad&p\neq q,\\
\exp\left(-\lambda^\frac{1}{q-1}t\right),\quad&q=p.
\end{cases}
$$ 
In particular, there is finite extinction for $p<q$, exponential decay for $p=q$, and algebraic decay in infinite time for $p>q$. Consequently, we expect that our theory translates directly to such situations.
\item A related problem is the study of \emph{rate-independent} problems \cite{mielke2005evolution} of the form 
$$\partial H(\partial_t u)+t\partial J(u)\ni 0,\qquad u(0)=f,$$
where now both $H$ and $J$ are absolutely 1-homogeneous. If $H$ equals a Hilbert space norm, a quick analysis shows that asymptotic profiles exist and are eigenfunctions. However, is that true for general~$H$? The analysis of this situation is extremely important since typically eigenvalue with respect to two 1-homogeneous functionals are of special interest, e.g., the $L^1$-eigenvalue problem for 1-Laplacian \cite{kawohl2007dirichlet,feld2019rayleigh}. 
\end{itemize}

\begin{acknowledgements}
This work was supported by the European Union’s Horizon 2020 research and innovation programme under the Marie Sk\l{}odowska-Curie grant agreement No 777826 (NoMADS). This is a pre-print of an article published in Journal of Evolution Equations. The final authenticated version is available online at: \url{https://doi.org/10.1007/s00028-019-00545-1}.
\end{acknowledgements}

\bibliographystyle{abbrv}
\bibliography{bibliography}  

\begin{appendix}
\section*{Appendix}
\section{Absolutely $p$-Homogeneous Convex Functionals and their Subdifferential}
\begin{proposition}\label{prop:null_space}
Let $J:\H\to\R\cup\{\infty\}$ be absolutely $p$-homogeneous and convex. Then it holds
\begin{enumerate}
\item $\calN(J):=\{u\in\H\st J(u)=0\}$ and $\dom(J):=\{u\in\H\st J(u)<\infty\}$ are linear subspaces of $\H$, referred to as \emph{null-space} and \emph{effective domain} of $J$.
\item $\calN(J)$ is closed if $J$ is lower semi-continuous.
\end{enumerate}
\end{proposition}

\begin{proposition}\label{prop:orth}
Under the conditions of Proposition~\ref{prop:null_space} it holds for any $u\in\H$ that $\partial J(u)\subset\calN(J)^\perp$.
\end{proposition}
\begin{proof}
Using the definition of the subdifferential \eqref{def:subdiff} for $v=\pm v_0\in\calN(J)$ together with \eqref{eq:euler} yields
$$|\langle\sg,v_0\rangle|\leq(p-1)J(u),\quad\forall\sg\in\partial J(u).$$
However, due to the homogeneity \eqref{eq:p-1-hom_subdiff} it holds $c^{p-1}\sg\in\partial J(c u)$ for all $c>0$. Replacing $\sg$ by $c^{p-1}\sg$ and $u$ by $cu$ in the inequality above yields
$$|\langle\sg,v_0\rangle|\leq c(p-1)J(u)$$
after using the homogeneity of $J$ and dividing by $c^{p-1}$. Letting $c\searrow 0$ concludes the proof.
\qed 
\end{proof}

The following Proposition states that the value of the functional $J$ and its subdifferential is invariant under addition of a null-space element.

\begin{proposition}\label{prop:subdiff_nullspace}
Let $v_0\in\calN(J)$ and $u\in\H$. If $J$ is lower semi-continuous, it holds $J(u)=J(u+v_0)$ and $\partial J(u)=\partial J(u+v_0)$.
\end{proposition}
\begin{proof}
In the case $p=1$ the statement has been proven in \cite{burger2016spectral,bungert2019nonlinear}. 

For $p>1$ one argues as follows: We claim that the convex conjugate $J^*$ meets $J^*(\sg)=\infty$ whenever $\sg\notin\calN(J)^\perp$. 
To see this we note that for any $v_0\in\calN(J)$ it holds
\begin{align*}
J^*(\sg)=\sup_{u\in\H}\langle\sg,u\rangle-J(u)\geq \langle\sg,v_0\rangle.
\end{align*}
If $\langle\sg,v_0\rangle\neq 0$ we can use that $\calN(J)$ is a linear space and replace $v_0$ by $cv_0$ for $c\in\R$ to obtain that $J^*(\sg)=\infty$.
Using this together with the fact that $J$ is lower-semicontinuous and therefore equals its biconjugate, we obtain
\begin{align*}
J(u)=J^{**}(u)=\sup_{\sg\in\H}\langle\sg,u\rangle - J^*(\sg) = \sup_{\sg\in\calN(J)^\perp}\langle\sg,u\rangle - J^*(\sg).
\end{align*}
From here it is obvious that $J(u+v_0)=J(u)$ for all $v_0\in\calN(J)$.

Having established this it is straightforward to show $\partial J(u)=\partial J(u+v_0)$ using the definition of the subdifferential together with the orthogonality from Proposition~\ref{prop:orth}.
\qed 
\end{proof}

\section{Gradient Flow of Absolutely $p$-Homogeneous Functionals}
We conclude the appendix by listing several important properties of the gradient flow \eqref{gradflow}, starting with the existence theorem due to Brezis. To this end we have to introduce the single-valued operator 
\begin{align}\label{eq:subgrad_minimal_norm}
\partial^0 J(u)=\argmin\left\lbrace\norm{\sg}\st\sg\in\partial J(u)\right\rbrace,\quad u\in\H,
\end{align}
which gives the subgradient with minimal norm in $\partial J(u)$ and is well-defined since the latter is a convex set.

\begin{theorem}[Brezis]\label{thm:brezis}
Let $J:\H\to\R\cup\{\infty\}$ be convex and lower semicontinuous and let $f\in\overline{\dom(J)}$. Then there exists exactly one continuous map $u:[0,\infty)\to\H$ which is Lipschitz continuous on $[\delta,\infty),\;\delta>0$ and right-differentiable on $(0,\infty)$ such that 
\begin{itemize}
\item $u(0)=f$,
\item $\sg(t):=\partial_t^+ u(t)=-\partial^0 J(u(t))$ for all $t>0$,
\item $t\mapsto J(u(t))$ is convex, non-increasing and Lipschitz continuous on $[\delta,\infty),\;\delta>0$ with
\begin{align}\label{eq:derivative_J}
\frac{\d^+}{\d t}J(u(t))=-\norm{\sg(t)}^2,\quad t>0,
\end{align}
\end{itemize}
where $\d^+/\d t$ and $\partial_t^+$ denote right-derivatives and will be replaced by standard derivative symbols throughout the rest of this manuscript.
\end{theorem}

\begin{proposition}[Conservation of mass]\label{prop:conserv_mass}
Let $u(t)$ solve the gradient flow \eqref{gradflow} with data $f$ and let $\overline{\cdot}:\H\to\calN(J)$ denote the orthogonal projection onto $\calN(J)$. Then it holds $\overline{u(t)}=\overline{f}$.
\end{proposition}
\begin{proof}
It holds
$$u(t)-f=-\int_0^t\sg(s)\d s$$
and from Proposition~\ref{prop:orth} we deduce $\overline{\sg(s)}=0$ for all $s>0$. Hence, also $\overline{u(t)-f}=0$ which implies the statement due to linearity of the projection.
\qed 
\end{proof}

\begin{proposition}\label{prop:wlog_orth_data}
Let $u(t)$ solve \eqref{gradflow} with data $f-\overline{f}\in\calN(J)^\perp$. Then $v(t):=u(t)+\overline{f}$ solves \eqref{gradflow} with data $f$.
\end{proposition}
\begin{proof}
The proof reduces to checking whether $-\partial_t v(t)=-\partial_t u(t)\in\partial J(v(t))$, which is true due to Proposition~\ref{prop:subdiff_nullspace}.
\qed 
\end{proof}

\begin{proposition}\label{prop:properties_GF}
Let $u(t)$ denote the solution of the gradient flow \eqref{gradflow} corresponding to the absolutely $p$-homogeneous functional $J$ and let $f\in\calN(J)^\perp$. Then it holds
\begin{alignat}{2}
\label{conv:u-to-0}
u(t)\to 0,\quad&J(u(t))\to 0,\quad&&t\to\infty,\\
\label{eq:diss_u}
\frac{\d }{\d t}\frac{1}{2}\norm{u(t)}^2&=-pJ(u(t)),\quad&&t>0.
\end{alignat}
\end{proposition}
\begin{proof}
The proof for $u(t)\to 0$ can be found in \cite{bungert2019nonlinear} and mainly relies on Proposition~\ref{prop:orth}. The statement $J(u(t))\to 0$ is classical \cite{brezis1973ope}. For \eqref{eq:diss_u} one uses the chain rule together with \eqref{eq:euler} to obtain
$$\frac{\d}{\d t}\frac{1}{2}\norm{u(t)}^2=\left\langle\partial_tu(t),u(t)\right\rangle=-\langle\sg(t),u(t)\rangle=-pJ(u(t)).$$
\qed 
\end{proof}
\end{appendix}

\end{document}